\documentclass[10pt]{amsart}
\usepackage{amsmath,amscd}
\usepackage{amsbsy}
\usepackage{amssymb}
\usepackage{amscd,amsthm}
\usepackage[all,cmtip]{xy}
\usepackage[russian,english]{babel}
\usepackage[utf8]{inputenc}
\usepackage{xcolor}

\usepackage{hyperref}

\renewcommand{\epsilon}{\varepsilon}

\renewcommand{\ell}{x}

\newtheorem{thm}{Theorem}\numberwithin{thm}{section}
\newtheorem{lem}[thm]{Lemma}

\newtheorem*{con2}{Conjecture}

\newtheorem*{rema2}{Remark}

\begin{document}
\begin{center}
\huge{Diophantine equations involving double factorials}\\[1cm] 
\end{center}
\begin{center}

\large{Sa$\mathrm{\check{s}}$a Novakovi$\mathrm{\acute{c}}$}\\[0,5cm]
{\small October 2025}\\[0,5cm]
\end{center}
{\small \textbf{Abstract}. 
We are motivated by a result of Alzer and Luca who presented all the integer solutions to the relations $(k!)^n-k^n=(n!)^k-n^k$ and $(k!)^n+k^n=(n!)^k+n^k$. We modify the equations by considering the double factorial instead and present all integer solutions. We also consider some variations of these equations. Furthermore, we study equations of the form $f(x)=A_1^{n_1}n_1!!\cdots A_r^{n_r}n_r!!$, where $f(x)$ is a rational polynomial, and show that under the ABC conjecture there are only finitely many integer solutions.}
\begin{center}
	\tableofcontents
\end{center}
\section{Introduction}
The theory of Diophantine equations has a long and rich history and has attracted the attention of many mathematicians. In particular, the study of diophantine equations involving factorials have been studied extensively. For example Brocard \cite{BR}, and independently Ramanujan \cite{RA}, asked to find all integer solutions for $n!=x^2-1$. It is still an open problem, known as Brocard's problem, and it is believed that the equation has only three solutions $(x,n)=(5,4), (11,5)$ and $(71,7)$. Overholt \cite{O} observed that a weak form of Szpiro's-conjecture implies that Brocard's equation has finitely many integer solutions. Some further examples of similar equations are:
\begin{itemize}
	\item[1)] $n!=x^k\pm y^k$ and $n!\pm m!=x^k$, see \cite{EO}.
	\item[2)] $\phi(x)=n!$, where $\phi$ is the Euler totient function \cite{FL}.
	\item[3)] $p(x)=m!$, where $p(x)\in\mathbb{Z}[x]$ \cite{L}.
	\item[4)] $\alpha\,m_1!_{S_1}\cdots m_r!_{S_r}=f(n!)$, where $f$ is an arithmetic function and $m_i!_{S_i}$ are certain Bhargava factorials \cite{BN}.
\end{itemize}
For the equations 1) and 4), it was shown that the number of integer solutions is finite. The equation in 3) has finitely many integer solutions, provided the ABC-conjecture holds, and 2) does have infinitely many solutions. There are a lot of more diophantine equations involving factorials and ploynomials that have been studied and we refer the interested reader to \cite{BN}, \cite{NO} and the references therein. For example, Alzer and Luca \cite{AL} considered the equations  $(k!)^n-k^n=(n!)^k-n^k$ and $(k!)^n+k^n=(n!)^k+n^k$ and presented all the integer solutions. Their results motivated us to study the diophantine relations
$$
(n!!)^{k}-n^k=(k!!)^{n}-k^n \quad \textnormal{and} \quad (n!!)^k+n^k=(k!!)^n+k^n
$$
and some variants of them (see Theorems 1.3, 1.4 and 1.5). Here $n!!$ denotes the double factorial. There are certainly more possible variants of the before mentioned equations that can be solved in a similar way as presented in the present paper. However, we have only focused on the considred ones. We want to mention that a crucial ingredient for the proofs of Theorems 1.1-1.5 is Lemma 2.1. We believe that this is known to the experts but could not find a reference for it. Therefore we give a proof.\\ 

\noindent
Our results are the following theorems.
\begin{thm}
	Let $n$ and $k$ be positive integers. The equation
	$$
	(n!!)^{k}-n^k=(k!!)^{n}-k^n
	$$
	holds if and only if $k=n$ or $(k,n)=(1,2), (2,1), (1,3), (3,1), (2,3), (3,2)$.
\end{thm}

\begin{thm}
	Let $n$ and $k$ be positive integers. The equation
	$$
	(n!!)^k+n^k=(k!!)^n+k^n
	$$
	holds if and only if $k=n$.
\end{thm}
\begin{thm}
	Let $n$ and $k$ be positive integers. The equation
	$$
	(n!!)^{k}-n^{k!!}=(k!!)^{n}-k^{n!!}
	$$
	holds if and only if $k=n$ or $(k,n)=(1,2), (2,1), (1,3), (3,1), (2,3), (3,2)$.
\end{thm}
\begin{thm}
	Let $n$ and $k$ be positive integers. The equation
	$$
	(n!!)^{k}-n^{k!}=(k!!)^{n}-k^{n!}
	$$
	holds if and only if $k=n$ or $(k,n)=(1,2), (2,1), (1,3), (3,1), (2,3), (3,2)$.
\end{thm}
\begin{thm}
	Let $n$ and $k$ be positive integers. The equation
	$$
	(n!!)^{k!!}+n^k=(k!!)^{n!!}+k^n
	$$
	holds if and only if $k=n$.
\end{thm}
We also study equations of the form $f(x)=A_1^{n_1}n_1!!\cdots A_r^{n_r}n_r!!$, where $f(x)$ is a rational polynomial and $A_1,...,A_r$ fixed positive integers. Equations of this type where $n!!$ is replaced by $n!$ have been considered by the author in \cite{NO}. In this context we found Theorems 1.6 and 1.7. Both theorems are in some sense a generalization of \cite{MU}, Theorem 3.1, where the author, among others, considers $x^2-!=n!!$ and uses the Hall conjecture to show that there are finitely many integer solutions. The proofs of Theorem 1.6 and Theorem 1.7 are similar to proofs of the analogous statements in \cite{NO}. Nontheless, we give the slightly modified arguments for convenience of the reader.
\begin{thm}
	Fix a non-zero integer $b$ and positive integers $A_1,...,A_r$. If $d>r$, then the equation $bn_1!!A_1^{n_1}\cdots n_r!!A_r^{n_r}=x^d$ has only finitely many integer solutions. If $d\leq r$, then $bn_1!!A_1^{n_1}\cdots n_r!!A_r^{n_r}=x^d$ has infinitely many integer solutions, except when $b<0$ and $d$ is even, where there are no solutions.
\end{thm}
\begin{rema2}
The statement of Theorem 1.6 remains true if $x^d$ is replaced by $ax^d$ with fixed positive rational number $a$ and if $A_i^{n_i}$ is replaced by $A_i^{n_i!}$ or $A_i^{n_i!!}$.	
\end{rema2}
We recall the ABC-conjecture which can be found for instance in \cite{LA}. For a non-zero integer $a$, let $N(a)$ be the \emph{algebraic radical}, namely $N(a)=\prod_{p|a}{p}$. Note that 
\begin{eqnarray}
	N(a)=\prod_{p|a}{p}\leq \prod_{p\leq a}{p}< 4^a,
\end{eqnarray}
where the last inequality follows from a Chebyshev-type result in elementary prime number theory and is called the Finsler inequality.
\begin{con2}[ABC-conjecture]
	For any $\epsilon >0$ there is a constant $K(\epsilon)$ depending only on $\epsilon$ such that whenever $A,B$ and $C$  are three coprime and non-zero integers with $A+B=C$, then 
	\begin{eqnarray}
		\mathrm{max}\{|A|,|B|,|C|\}<K(\epsilon)N(ABC)^{1+\epsilon}
	\end{eqnarray}
	holds.
\end{con2} 
\begin{thm}
	Let $f(x)\in\mathbb{Q}[x]$ be a polynomial of degree $d\geq 2$ which is not monomial and has at least two distinct roots. Fix a non-zero integer $b$ and positive integers $A_1,...,A_r$. Then the ABC-conjecture implies that $bn_1!!A_1^{n_1}\cdots n_r!!A_r^{n_r}=f(x)$ has only finitely many integer solutions with $n_i>0$.  
\end{thm}
\begin{rema2}
	\textnormal{The statement of Theorem 1.7 remains valid if $A_i^{n_i}$ is replaced by $A_i^{n_i!}$ or $A_i^{n_i!!}$.}  
\end{rema2}
\section{Proof of Theorem 1.1}
\noindent
For the proof, we need the following observation:
\begin{lem}
	Let $n\geq 2$ be an integer. Then $n!!^{\frac{1}{n}}$ is a strictly increasing sequence.
\end{lem}
\begin{proof}
	We show that
	$$
	\frac{n!!^{\frac{1}{n}}}{(n-1)!!^{\frac{1}{n-1}}}>1.
	$$
	This is equivalent to show 
	$$
	\frac{n!!^{n-1}}{(n-1)!!^n}>1.
	$$
	Note that
	$$
	\frac{n!!^{n-1}}{(n-1)!!^n}=(\frac{n!!}{(n-1)!!})^{n-1}\cdot \frac{1}{(n-1)!!}.
	$$
	Since $\frac{m+1}{m}>\sqrt{\frac{m+2}{m}}$ for $m\geq 1$, we obtain for $n$ even:
	$$
	\frac{n!!}{(n-1)!!}=\frac{2}{1}\cdot \frac{4}{3}\cdot \frac{6}{5}\cdots \frac{n}{n-1}>\sqrt{\frac{3}{1}}\cdot \sqrt{\frac{5}{3}}\cdots \sqrt{\frac{n+1}{n-1}}=\sqrt{n+1}>\sqrt{n}.
	$$
	Therefore
	$$
	\frac{n!!^{n-1}}{(n-1)!!^n}>\frac{(\sqrt{n})^{n-1}}{(n-1)!!}.
	$$
	It remains to show that $(\sqrt{n})^{n-1}>(n-1)!!$. But this follows from 
	$$
	(\sqrt{n})^{n-1}=n^{(n-2)/2}\cdot \sqrt{n}>(n-1)(n-3)\cdots 1, 
	$$
	since on the right hand side there are $\frac{n-2}{2}$ factors $>1$, because $n$ is assumed to be even.
	If $n$ is odd, we obtain
	$$
		\frac{n!!}{(n-1)!!}=\frac{3}{2}\cdot \frac{5}{4}\cdot \frac{7}{6}\cdots \frac{n}{n-1}>\sqrt{\frac{4}{2}}\cdot \sqrt{\frac{6}{4}}\cdots \sqrt{\frac{n+1}{n-1}}=\sqrt{\frac{n+1}{2}}.
	$$
	Again, we have
	$$
	(\frac{n!!}{(n-1)!!})^{n-1}>(\sqrt{\frac{n+1}{2}})^{n-1}
	$$
	and it remains to show that
	$$
	(\sqrt{\frac{n+1}{2}})^{n-1}>(n-1)!!.
	$$
	Essentially, this follows from the AM-GM inequality. To be precise, since $n$ is odd, we write $n=2k+1$. This gives
	$$
	(\frac{n!!}{(n-1)!!})^{n-1}=\sqrt{k+1}^{2k}=(k+1)^k
	$$
	and we have to prove that
	$$
	(k+1)^k>(2k)!!=2^kk!.
	$$
	But this is equivalent to
	$$
	(\frac{k+1}{2})^k>k!.
	$$
	Now applying AM-GM inequality yields
	$$
	\frac{1+2+3+\cdots + k}{k}=\frac{k(k+1)}{2k}=\frac{k+1}{2}>\sqrt[k]{k!}.
	$$
	This completes the proof.
\end{proof}
\noindent
We now prove Theorem 1.1. Obviously, if $k=n$, the equation $(n!!)^{k}-n^k=(k!!)^{n}-k^n$ is valid. By symmetry we may assume that $n>k$. We consider three cases.\\
\noindent
the case $k=1$:\\
\noindent
If $k=1$, the equation becomes $n!!-n=0$. But this holds only for $n=1$, $n=2$ and $n=3$.\\
\noindent
the case $k=2$:\\
\noindent
In this case, the equation reduces to $n!!^2-n^2=0$. This is equivalent to $n!!=n$ and was treated before in the case $k=1$.\\
\noindent
the case $k\geq 3$:\\
\noindent
According to Lemma 2.1, the sequence $n!!^{\frac{1}{n}}$ is strictly increasing for $n>3$. This implies $n!!^k>k!!^n$. Notice that $-n^{\frac{1}{n}}$ is also strictly increasing for $n>3$. Hence $-n^k>-k^n$. Adding up the two inequalities yields
	$$
	n!!^k-n^k>k!!^n-k^n.
	$$
	This completes the proof of Theorem 1.1.
	\section{Proof of Theorem 1.2}
	\noindent
	Obviously, the equation $(n!!)^k+n^k=(k!!)^n+k^n$ holds for $k=n$. We show that $(n!!)^k+n^k=(k!!)^n+k^n$ implies $k=n$. By symmetry, we may assume $n\geq k$. There are four cases to consider.\\
	\noindent
	the case $k=1$:\\
	\noindent
	The equation reads $n!!+n=2$. But this holds only for $n=1$.\\
	\noindent
	the case $k=2$:\\
	\noindent
	Then the equation becomes $n!!^2+n^2=2^{n+1}$. This holds only for $n=2$. If $n=3$, we have $3^2+3^2\neq 2^4$. If $n>3$, we can rewrite the equation as $2^{n+1}=n^2((n-2)^2(n-4)^2\cdots +1)$. It is easy to see that this equation has no solution for $n>3$, since one of the factors on the right hand side must be odd and is therefore never a power of $2$.\\
	\noindent
	the case $k=3$:\\
	\noindent
	In this case, we have $n!!^3+n^3=2\cdot 3^n$. Since $n\geq k$, we have $n\geq 3$. If $n=3$, we have $3!!^3+3^3=2\cdot 3^3$. If $n\geq 4$, we rewrite the equation as
	$$
	n^3((n-2)^3(n-4)^3\cdots +1)=2\cdot 3^n.
	$$
	If $n\geq 4$ is even, the factor $((n-2)^3(n-4)^3\cdots +1)$ is odd. But this implies $n^3=2$ which is impossible. If $n\geq 4$ is odd, we conclude that $n^3=3^a$ for some $a\geq 2$ with $a\leq n$. But this means $(n-2)!!^3+1=2\cdot 3^b$ with an positive integer $b<n$. If $b=0$, we get $a=n$ and hence $n^3=3^n$ which is possible only for $n=3$. So we may assume $0<b<n$. Notice that $n$ odd with $n\geq 4$ means $n\geq5$. For $n=5$, we get $(5-3)!!^3+1=3^3+1<2\cdot 3^4$. We claim that $(n-2)!!^3>2\cdot 3^{n-1}$ for $n\geq7$ and $n$ odd. We use a simple induction argument. We start with $n=7$ and see that $(7-2)!!^3=15^3>2\cdot 3^6$. Since $n-1$ is even, we can use the proof of Lemma 2.1 where we established
	$$
	\frac{(n-1)!!}{(n-2)!!}>\sqrt{n}
	$$
	and hence
	$$
	(\frac{(n-1)!!}{(n-2)!!})^3>(\sqrt{n})^3>3
	$$
	for $n\geq 3$. Equivalently, $(n-1)!!^3>3\cdot (n-2)!!^3$. Now by induction hypothesis, we obtain
	$$
	(n-1)!!^3>3\cdot (n-2)!!^3>2\cdot 3^{n}.
	$$
	And since $b<n$, we have $(n-2)!!^3+1>(n-2)!!^3>2\cdot 3^{n-1}\geq 2\cdot 3^b$.\\
	\noindent
	the case $k\geq 4$:\\
	\noindent
	Note that $n\geq k\geq 4$. 
	But this implies there exists a rational number $b\geq 1$ such that $n=bk$. To continue, we assume $b>1$ and produce a contradiction. Since $(n^{1/n})_{n=3}^{\infty}$ is strictly increasing and since $n>k$ as $b>1$, we have $k^n-n^k>0$. Thus 
	$$
	k^n-n^k=n!!^k-k!!^{bk}>0
	$$
	and therefore $n!!-k!!^b\geq 1$. It follows
	$$
	k^{bk}=k^n>k^n-n^k=n!!^k-k!!^{bk}=(n!!-k!!^b)\cdot \sum_{j=0}^{k-1}n!!^jk!!^{b(k-1-j)}>k!!^{b(k-1)}.
	$$
	This means
	$$
	k^k>k!!^{k-1}
	$$
	or equivalently 
	$$
	k^{k/(k-1)}>k!!.
	$$
	But this is false, since for $k\geq 4$ one has
	$$
	k^{k/(k-1)}\leq k^{4/3}=k\cdot \sqrt[3]{k}<k\cdot (k-2)\leq k\cdot (k-2)\cdot (k-4)\cdots =k!!.
	$$
	This gives a contradiction and therefore $b=1$. This completes the proof. 
	\section{Proof of Theorems 1.3 and 1.4}
	\noindent
	We first prove Theorem 1.3. If $k=n$, the equation $n!!^k-n^{k!!}=k!!^n-k^{n!!}$ is valid. By symmetry we may assume $n>k$ and consider three cases.\\
	\noindent
	the case $k=1$:\\
	\noindent
	In this case the equation becomes $n!!-n=0$. The only solution are $n=1$, $n=2$ or $n=3$.\\
	\noindent
	the case $k=2$:\\
	\noindent
	In this case the equation is $n!!^2-n^2=2^n-2^{n!!}$. One easily verifies that $n=1$, $n=2$ are $n=3$ are solutions. For $n\geq 4$, we have $n!!>n$ and therefore the left and side of the equation is positiv whereas the right hand side is negative. Hence there is no solution if $n\geq 4$.\\
	\noindent
	the case $k\geq 3$:\\
	\noindent
	From Lemma 2.1, we conclude $n!!^k>k!!^n$. Note that the sequence $-n^{1/n}$ is strictly increasing for $n>3$. From this we obtain $-n^{1/n}>-k^{1/k}$ or equivalently $k^{1/k}>n^{1/n}$. But this implies
	$$
	(k^{1/k})^{k\cdot n!!}=k^{n!!}>(n^{1/n})^{k\cdot n!!}=n^{k\cdot (n-2)!!}>n^{k\cdot (k-2)!!}=n^{k!!}
	$$
	since $n-2>k-2$. This shows $-n^{k!!}>-k^{n!!}$. Adding up the inequalities yields
	$$
	n!!^k-n^{k!!}>k!!^n-k^{n!!}.
	$$
	This completes the proof.\\
	\noindent
	We now prove Theorem 1.4. The proof is exactly the same as the proof of Theorem 1.3. The only minor difference appears for $k\geq 3$. For convenience of the reader, we state the argument: from Lemma 2.1, we conclude $n!!^k>k!!^n$. The sequence $-n^{1/n}$ is strictly increasing for $n>3$ and therefore $-n^{1/n}>-k^{1/k}$ or equivalently $k^{1/k}>n^{1/n}$. But this implies
	$$
	(k^{1/k})^{k\cdot n!}=k^{n!}>(n^{1/n})^{k\cdot n!}=n^{k\cdot (n-1)!}>n^{k\cdot (k-1)!}=n^{k!}
	$$
	since $n-1>k-1$. This shows $-n^{k!}>-k^{n!}$. Adding up the inequalities yields
	$$
	n!!^k-n^{k!}>k!!^n-k^{n!}.
	$$
	\section{Proof of Theorem 1.5}
	\noindent
	The equation $n!!^{k!!}+n^k=k!!^{n!!}+k^n$ holds for $k=n$. We now show that $n!!^{k!!}+n^k=k!!^{n!!}+k^n$ implies $k=n$. By symmetry we assume $n\geq k$ and consider four cases.\\
	\noindent
	the case $k=1$:\\
	\noindent
	Our equation becomes $n!!+n=2$. But this holds only for $n=1$.\\
	\noindent
	the case $k=2$:\\
	\noindent
	If $k=2$, the equation is $n!!^2+n^2=2^{n!!}+2^n$. By an easy induction argument one can show that $2^m>m^2$ for all $m\geq 5$. From this, we obtain that $2^n+2^{n!!}>n^2+n!!^2$ for $n\geq 5$. Hence there is no solution for $n\geq 5$. By pluging in $n=1$, $n=2$, $n=3$ and $n=4$, we get that $n=2$ is the only solution.\\
	\noindent
	the case $k=3$:\\
	\noindent
	Now the equation becomes $n!!^3+n^3=3^{n!!}+3^n$. Pluging in $n=3$ gives $3!!^3+3^3=3^{3!!}+3^3$. Now by a simple induction argument we get $3^m>m^3$ for $m\geq 4$. And this implies $3^n+3^{n!!}>n^3+n!!^3$ for $n\geq 4$.\\
	\noindent
	the case $k\geq 4$:\\
	\noindent
	Note that $n\geq k\geq 4$. 
	But this implies that there exists a rational number $b\geq 1$ such that $n=bk$. To continue, we assume $b>1$ and produce a contradiction. Since $(n^{1/n})_{n=3}^{\infty}$ is strictly increasing and since $n>k$ as $b>1$, we have $k^n-n^k>0$. Thus 
	$$
	k^n-n^k=n!!^{k!!}-k!!^{(bk)!!}>0.
	$$
	Note that there is a rational number $c\geq 1$ such that $(bk)!!=c\cdot k!!$ and that, since we assumed $b>1$, we have $c>b>1$. 
	Thus $n!!-k!!^c\geq 1$. It follows
	$$
	k^{bk}=k^n>k^n-n^k=n!!^{k!!}-k!!^{c\cdot k!!}=(n!!-k!!^c)\cdot \sum_{j=0}^{k!!-1}n!!^jk!!^{c(k!!-1-j)}>k!!^{c(k!!-1)}.
	$$
	Since $k!!-1>k-1$ and $c>b$, we obtain
	$$
	k^{bk}>k!!^{c(k!!-1)}>k!!^{b(k-1)}
	$$
	or equivalently 
	$$
	k^{k/(k-1)}>k!!.
	$$
	But this is false, since for $k\geq 4$ one has
	$$
	k^{k/(k-1)}\leq k^{4/3}=k\cdot \sqrt[3]{k}<k\cdot (k-2)\leq k\cdot (k-2)\cdot (k-4)\cdots =k!!.
	$$
	This gives a contradiction and therefore $b=1$. This completes the proof.

	\section{Proof of Theorems 1.6}
	\noindent
	We give the proof only for the cases $r=1,2$, since the arguments for arbitrary $r>2$ are analogous. So let $r=1$. We consider the equation
	\begin{center}
		$bn!!A^n=x^d$.
	\end{center}
	If $n$ is odd and $n>2\mathrm{max}\{A,|b|\}$, then there is a prime number $p$ in the interval $(n/2,n)$ wich is larger than $\mathrm{max}\{A,|b|\}$. The prime $p$ will appear with exponent one in $bn!!A^n$. Since $d>1$, the number $bn!!A^n$ cannot be a perfect power. If $n=2l$ is even and  $l>2\mathrm{max}\{A,|b|,2\}$, then there is a prime number $p$ in the interval $(l/2,l)$ wich is larger than $\mathrm{max}\{A,|b|,2\}$. The prime $p$ will appear with exponent one in $bn!!A^n$.\\
	\noindent
	Now let $r=2$. We consider
	\begin{eqnarray}
		b n!! A^n m!! B^m=x^d.
	\end{eqnarray} 
	Without loss of generality, we assume $n,m$ to be odd. If $n=2l$ or $m=2s$ are even, we consider $l$ repectively $s$ instead of $n$ or $m$. This is also pointed out in the proof for the case $r=1$ above.\\
	\noindent
	If $n>2\mathrm{max}\{A,B,|b|\}$, then there is a prime number $p$ in the interval $(n/2,n)$ which is larger than $\mathrm{max}\{A,B,|b|\}$. There are three cases to consider. 
	\begin{itemize}
		\item[1)] $n>m$. In this case we see that the prime $p$ will appear with exponent at most two in the product $b n!! A^n m!! B^m$. Since $d\neq 2$, the product $b n!! A^n m!! B^m$ cannot be of the form $x^d$. Therefore, there are no integer solutions if $n>2\mathrm{max}\{A,B,|b|\}$ and $n>m$.
		\item[2)] $n<m$. In this case $m>2\mathrm{max}\{A,B,|b|\}$. Then there is a prime number $p$ in the interval $(m/2,m)$ which is larger than $\mathrm{max}\{A,B,|b|\}$. Again, this prime $p$ will appear with exponent at most two in the product $b n!! A^n m!! B^m$. Since $d\neq 2$, the product $b n!! A^n m!! B^m$ cannot be of the form $x^d$. Hence, there are no integer solutions if $n>2\mathrm{max}\{A,B,|b|\}$ and $n<m$.
		\item[3)] $n=m$. In this case equation (4) becomes
		\begin{eqnarray}
			b(n!!)^2(AB)^n=x^d.
		\end{eqnarray}
		As $n>2\mathrm{max}\{A,B,|b|\}$, the prime number $p$ in the interval $(n/2,n)$ will appear with exponent two in the product $b(n!!)^2(AB)^n$. Since $d\neq 2$, the product $b(n!!)^2(AB)^n$ cannot be of the form $x^d$. Therefore, there are no integer solutions if $n>2\mathrm{max}\{A,B,|b|\}$ and $n=m$.
	\end{itemize}
	So we are left with $n\leq 2\mathrm{max}\{A,B,|b|\}$. Now we consider the following cases.
	\begin{itemize}
		\item[1)] $m<n$. In this case there can be only finitely many integer solutions $(n,m,z)$ satisfying (4).
		\item[2)] $n=m$. In this case there can be only finitely many integer solutions $(n,n,z)$ satisfying (4).
		\item[3)] $n<m$. In this case, either we must have $n<m\leq2\mathrm{max}\{A,B,b\}$ or $n\leq 2\mathrm{max}\{A,B,b\}<m$. Clearly, if $n<m\leq2\mathrm{max}\{A,B,|b|\}$, there can be only finitely many integer solutions $(n,m,z)$ satisfying (4). Now if $n\leq 2\mathrm{max}\{A,B,|b|\}<m$, we conclude from 2) from above that there is a prime number $p$ in the interval $(m/2,m)$ which is larger than $\mathrm{max}\{A,B,|b|\}$. This prime $p$ will appear with exponent at most two in the product $b n!! A^n m!! B^m$. Since $d\neq 2$, the product $b n!! A^n m!! B^m$ cannot be of the form $x^d$. 
	\end{itemize}
	Summarizing, we see that the equation (4) can have only finitely many integer solutions. This completes the first part of the proof.\\
	\noindent
	Now let us consider the case $d\leq r$. Obviously, there are infinitely many integer solutions for $x=bn_1!!A_1^{n_1}n_2!!A_2^{n_2}\cdots n_r!!A_r^{n_r}$. We therefore assume $d\geq2$. Notice that in case $b<0$ and $d$ is even the equation has no solution. So we consider the equation
	\begin{center}
		$x^d=bn_1!!A_1^{n_1}n_2!!A_2^{n_2}\cdots n_r!!A_r^{n_r}$,
	\end{center}
	where $b>0$ and $d$ arbitrary or $b<0$ and $d$ odd. 
	Since $d\leq r$, we can rewrite the equation as
	\begin{center}
		$x^d=bn_1!!A_1^{n_1}n_2!!A_2^{n_2}\cdots n_d!!A_d^{n_d}\cdot(n_{d+1}!!A_{d+1}^{n_{d+1}}\cdots n_r!!A_r^{n_r})$.
	\end{center}
	Now we set $n_{d+1}=n_{d+2}=\cdots =n_r=1$ and $n_1=n_2=\cdots =n_{d-1}=m$ and $n_d=m+2$. Then the equation becomes
	\begin{center}
		$x^d=(A_1\cdots A_d)^m\cdot (m!!)^d\cdot bA_d^2\cdot A_{d+1}\cdots A_r\cdot(m+2)$.
	\end{center}
	We rewrite again:
	\begin{center}
		$x^d=(A_1\cdots A_d)^{(m-(d-2))}\cdot (m!!)^d\cdot b\cdot(A_1\cdots A_d)^{(d-2)}\cdot A_d^2\cdot A_{d+1}\cdots A_r\cdot(m+2)$.
	\end{center}
	Now we want to choose $m$ such that $m-(d-2)=m-d+2=ds$. This is equivalent to $m+2=d(s+1)$. If $b>0$, we set
	\begin{center}
		$R:=b\cdot(A_1\cdots A_d)^{(d-2)}\cdot A_d^2\cdot A_{d+1}\cdots A_r$
	\end{center}
	Then the above equation becomes 
	\begin{center}
		$x^d=(A_1\cdots A_d)^{(m-(d-2))}\cdot (m!!)^d\cdot Rd\cdot(s+1)$.
	\end{center}
	Now we can set $s=(Rd)^{td-1}-1$ where $t>0$ is any positive integer and choose $m$ such that  $m-(d-2)=d\cdot((Rd)^{td-1}-1)$ is a multiple of $d$. Notice that $d\geq2$ by assumption and hence $s\geq 1$. Our diophantine equation becomes
	\begin{center}
		$x^d=((A_1\cdots A_d)^{((Rd)^{td-1}-1)})^d\cdot ((d(Rd)^{td-1}-2)!!)^d\cdot ((Rd)^t))^d$.
	\end{center}
	This shows that we can find infinitely many interger solutions $(x,n_1,...,n_r)$ with $n_i>0$. If $b<0$ and $d$ odd, we set 
	\begin{center}
		$R':=|b|\cdot(A_1\cdots A_d)^{(d-2)}\cdot A_d^2\cdot A_{d+1}\cdots A_r$. 
	\end{center}
	Then $R=(-1)R'$ and the diophantine equation becomes
	\begin{center}
		$x^d=((A_1\cdots A_d)^{((Rd)^{td-1}-1)})^d\cdot ((d(Rd)^{td-1}-1)!!)^d\cdot ((Rd)^t))^d\cdot (-1)^{td}$.
	\end{center}
	Choosing $t$ odd, we conclude that $td$ is odd. Thus $td-1$ is even. Since $d\geq2$ it follows that $(Rd)^{td-1}-1>0$. This shows that there are also infinitely many integer solutions in this case. Note that the solutions are constructed only using the fixed integers $b,A_1,...,A_r$ and the given degree $d$. This completes the proof.
	\begin{rema2}
		\textnormal{As mentioned in the introduction, the statement of the theorem remains true if $x^d$ is replaced by $ax^d$ with fixed positive rational number $a$ and if $A_i^{n_i}$ is replaced by $A_i^{n_i!}$ or $A_i^{n_i!!}$. To prove that the first part of the theorem remains valid after replacing is straight forward. We give an argument only for the second part of the statement, namely if all the $A_i^{n_i}$ are replaced by $A_i^{n_i!}$. We give the argument only for $d$ even, since the case $d$ odd is similar. So let $d\geq 2$ be even and assume that $b>0$. Since $d\leq r$, we rewrite our equation as 
		$$
	x^d=bA_1^{n_1!}n_1!\cdots A_d^{n_d!}n_d!\cdot (A_{d+1}^{n_{d+1}!}n_{d+1}!)\cdots A_r^{n_r!}n_r!).
$$
Now we set $n_{d+1}=n_{d+2}=\cdots =n_r=1$ and $n_1=\cdots =n_{d-1}=m$ and $n_d=m+2$. Then the equation becomes
$$
x^d=(m!!)^d(A_1\cdots A_d^{(m+2)(m+1)})^{m!}\cdot((m+2)bA_{d+1}\cdots A_r).
$$
Since $d\geq 2 $ we can set $m+2=b^{d-1}(A_{d+1}\cdots A_r)^{d-1}\cdot y^d$, with $y\in \mathbb{N}$ large enough such that $m>d$ and see that the equation has infinitely many integer solutions of the form
$$
x=m!!A_1^t\cdots A_d^{(m+2)(m+1)t}bA_{d+1}\cdots A_r\cdot y,
$$
where $t$ is such that $m!=t\cdot d$. }
	\end{rema2}
	\section{Proof of Theorem 1.7}
	\noindent
	Multiplying the equation $bn_1!!A_1^{n_1}\cdots n_r!!A_r^{n_r}=f(x)$ by a certain integer, we may assume that $f(x)$ is a polynomial with integer coefficients. So without loss of generality, we assume 
	\begin{eqnarray*}
		f(x)=a_0x^d+a_1x^{d-1}+...+a_d
	\end{eqnarray*}
	with $a_i\in \mathbb{Z}$. Now multiply the equation $bn_1!!A_1^{n_1}\cdots n_r!!A_r^{n_r}=f(x)$ by $d^da_0^{d-1}$. We obtain
	\begin{eqnarray*}
		y^d+b_1y^{d-1}+...+b_d=c(n_1!!A_1^{n_1}\cdots n_r!!A_r^{n_r})
	\end{eqnarray*}
	for a constant $c$, where $c=bd^da_0^{d-1}$ and $y:=a_0dx$. Notice that $b_i=d^ia_ia_0^{i-1}$ so that we can make the change of variable $z:=y+\frac{b_1}{d}$. Since we are assuming that $f(x)$ has at least two distinct roots, the change of variable produces a polynomial that does not have a monomial of degree $d-1$. Therefore we get the following equation
	\begin{eqnarray}
		z^d+c_2d^{d-2}+...+c_d=c(n_1!!A_1^{n_1}\cdots n_r!!A_r^{n_r}).
	\end{eqnarray}
	Notice that $c_i$ are integer coefficients wich can be computed in terms of $a_i$ and $d$. Now let $Q(X)=X^d+c_2X^{d-2}+...+c_d$ and notice that when $|z|$ is large one has
	\begin{eqnarray}
		\frac{|z|^d}{2}<|Q(z)|<2|z|^d.
	\end{eqnarray}
	For the rest of the proof we denote by $C_1,C_2,...$ computable positive constants depending on the coefficients $a_i$ and eventually on some small $\epsilon >0$ which comes into play later by applying the ABC-conjecture.
	
	Whenever $(n_1,...,n_r,z)$ is a solution to $n_1!!A_1^{n_1}\cdots n_r!!A_r^{n_r}=f(x)$ we conclude from (5) and (6) that there exist constants $C_1$ and $C_2$ such that
	\begin{eqnarray}
		|d\cdot\mathrm{log}|z|-\mathrm{log}(n_1!!A_1^{n_1}\cdots n_r!!A_r^{n_r})|<C_1,
	\end{eqnarray}
	for $|z|>C_2$ (see \cite{L} equation (10)). 
	Now let $R(X)\in \mathbb{Z}[X]$ be such that $Q(X)=X^d+R(X)$. Since $f(x)$ is not monomial and has at least two distinct roots, $R(X)$ can be assumed to be non-zero, let $j\leq d$ be the largest integer with $c_j\neq 0$. We rewrite (5) as
	\begin{eqnarray*}
		z^j+c_2z^{j-2}+...+c_j=\frac{c(n_1!!A_1^{n_1}\cdots n_r!!A_r^{n_r})}{z^{d-j}}.
	\end{eqnarray*}
	Let $R_1(X)$ be the polynomial 
	\begin{eqnarray*}
		R_1(X):= \frac{R(X)}{X^{d-j}}=c_2X^{j-2}+\cdots +c_j.
	\end{eqnarray*}
	It is shown in \cite{L} there are constants $C_3$ and $C_4\geq C_2$ such that
	\begin{eqnarray*}
		0<|R_1(z)|< C_3|z|^{j-2},
	\end{eqnarray*}
	for $|z|> C_4$. So we have 
	\begin{center}
		$z^j+R_1(z)=\frac{c(n_1!!A_1^{n_1}\cdots n_r!!A_r^{n_r})}{z^{d-j}}$. 
	\end{center}
	For $D=\mathrm{gcd}(z^j, R_1(z))$ we have
	\begin{eqnarray*}
		\frac{z^j}{D}+\frac{R_1(z)}{D}=\frac{c(n_1!!A_1^{n_1}\cdots n_r!!A_r^{n_r})}{z^{d-j}D}
	\end{eqnarray*}
	Applying the ABC-conjecture to $A=\frac{z^j}{D}$, $B=\frac{R_1(z)}{D}$ and $C=\frac{c(n_1!!A_1^{n_1}\cdots n_r!!A_r^{n_r})}{z^{d-j}D}$, we find
	\begin{eqnarray}
		\frac{|z|^j}{D}< C_5N(\frac{z^jR_1(z)c(n_1!!A_1^{n_1}\cdots n_r!!A_r^{n_r})}{D^3})^{1+\epsilon},
	\end{eqnarray}
	where $C_5$ depends only on $\epsilon$. It is shown in \cite{L}, p.272 that 
	\begin{eqnarray}
		N(\frac{|z|^j}{D})\leq |z|,\\
		N(\frac{R_1(z)}{D})<\frac{C_3|z|^{j-2}}{D}.
	\end{eqnarray}
	Moreover, we have
	\begin{eqnarray*}
		N(\frac{c(n_1!!A_1^{n_1}\cdots n_r!!A_r^{n_r})}{z^{d-j}D})\leq N(c)N(n_1!!A_1^{n_1}\cdots n_r!!A_r^{n_r})\leq N(c)N(n_1!!A_1^{n_1})\cdots N(n_r!!A_r^{n_r}).
	\end{eqnarray*}
	This gives 
	\begin{eqnarray*}
		N(\frac{c(n_1!!A_1^{n_1}\cdots n_r!!A_r^{n_r})}{z^{d-j}D})\leq N(c)N(n_1!!A_1^{n_1})\cdots N(n_r!!A_r^{n_r})\leq C_6N(n_1!!)\cdots N(n_r!!),
	\end{eqnarray*} 
	where $C_6=N(c)N(A_1)\cdots N(A_r)$. 
	From (1) it follows
	\begin{eqnarray}
		N(\frac{c(n_1!!A_1^{n_1}\cdots n_r!!A_r^{n_r})}{z^{d-j}D})<C_64^{n_1}\cdots 4^{n_r}=C_64^{(n_1+\cdots +n_r)}
	\end{eqnarray}
	and from (9), (10) and (11) we get
	\begin{eqnarray}
		N(\frac{|z|^j}{D})N(\frac{R_1(z)}{D})N(\frac{c(n_1!!A_1^{n_1}\cdots n_r!!A_r^{n_r})}{z^{d-j}D})<\frac{C_3C_6|z|^{j-1}4^{(n_1+\cdots +n_r)}}{D}.
	\end{eqnarray}
	From inequalities (8) and (12), we obtain
	\begin{eqnarray}
		\frac{|z|^j}{D}<C_7\bigl(\frac{|z|^{j-1}4^{(n_1+\cdots +n_r)}}{D}\bigr)^{(1+\epsilon)}
	\end{eqnarray}

	If we choose $\epsilon =\frac{1}{2d}\leq \frac{1}{2j}$, inequality (13) implies that
	\begin{eqnarray*}
		|z|^{1/2}<|z|^{1+\epsilon -\epsilon j}< C_84^{(n_1+...+n_r)(1+\epsilon)},
	\end{eqnarray*}
	or simply 
	\begin{eqnarray*}
		\mathrm{log}|z|<C_9n_1+...+C_9n_r+C_{10}.
	\end{eqnarray*}
	Thus
	\begin{eqnarray*}
		d\cdot\mathrm{log}|z|<C_{11}n_1+...+C_{11}n_r+C_{12}.
	\end{eqnarray*}
	This gives
	\begin{eqnarray}
		\mathrm{log}(n_1!!A_1^{n_1}\cdots n_r!!A_r^{n_r})<C_1+d\cdot \mathrm{log}|z|< C_{11}n_1+...+C_{11}n_r+C_{13}.
	\end{eqnarray}
	We can simplify (14) and finally obtain
	\begin{center}
		$\mathrm{log}(n_1!!)+\mathrm{log}(n_2!!)+\cdots +\mathrm{log}(n_r!!)<C_{14}n_1+C_{14}n_2+\cdots C_{14}n_r+C_{13}$.	
	\end{center}
	Now we can conclude that only finitely many $(n_1,...,n_r)$ satisfy (14). We give the argument for $r=2$. So lets consider an inequality of the form
	\begin{center}
		$\mathrm{log}(n!!)+\mathrm{log}(m!!)<A'n+B'm+C'$
	\end{center}
	where $A',B'$ and $C'$ are positive constant intergers. Assume there are infinitely many pairs $(n,m)$ of natural numbers satisfying the inequality. There are three cases: 
	\begin{itemize}
		\item[1)] infinitely many $n$ and finitely many $m$: let $s$ denote the maximum of these $m$ and $t$ the minimum. Then we have
		\begin{center}
			$\mathrm{log}(n!!)+\mathrm{log}(t!!)<A'n+B's+C'$
		\end{center}
		and therefore
		\begin{center}
			$\mathrm{log}(n!!)<A'n+E'$
		\end{center}
		where $E'=\mathrm{log}(t!!)+B's+C'$ is a constant. With Stirling's approximation for $n!!$ we find that there are only finitely many $n$ satisfying 
		\begin{center}
			$\mathrm{log}(n!!)<A'n+E'$.
		\end{center}
		This gives a contradition. 
		\item[2)] infinitely many $m$ and finitely many $n$: reverse the role of $n$ and $m$. 
		\item[3)] infinitely many $n$ and infinitely many $m$: the argument is similar. At some point Stirlings apprximation shows that $\mathrm{log}(n!!)$ exceeds $A'n$ and $\mathrm{log}(m!!)$ exceeds $B'm+C'$. Therefore infinitely many $n$ and infinitely many $m$ is impossible. 
	\end{itemize}
	Since there are only finitely many $(n_1,...,n_r)$ satisfying  (14), we finally conclude from (7) that $|z|< C_{15}$. This completes the proof.
	\begin{rema2}
		\textnormal{Since $N(A^{n!}n!!)<N(A)N(n!!)<N(A)\cdot 4^n$ and $N(A^{n!!}n!!)<N(A)N(n!!)<N(A)\cdot 4^n$, the statement of Theorem 1.7 remains valid if $A_i^{n_i}$ is replaced by $A_i^{n_i!}$ or $A_i^{n_i!!}$. It is even possible to mix all these factors and consider, for example, equations of the $f(x)=bA^{n!}n!B^{m}m!!C^{l!!}l!!$. The ABC-conjecture implies that there are finitely many integer solutions for these equations as well.}
	\end{rema2}

\vspace{0.3cm}
\noindent
{\tiny HOCHSCHULE FRESENIUS UNIVERSITY OF APPLIED SCIENCES 40476 D\"USSELDORF, GERMANY.}\\
E-mail adress: sasa.novakovic@hs-fresenius.de\\

\end{document}